\documentclass[11pt]{scrartcl}

\usepackage[showeqnr,theoremdefs,
final,
]{latexdev}

\usepackage{authblk}

\newcommand*\email[1]{\href{mailto:#1}{\nolinkurl{#1}}}
\newcommand*\thankmail[1]{}
\title{A classification of volume preserving generating forms in $\RR^3$}
\author[1]{Olivier Verdier\thankmail{olivier.verdier@math.umu.se}}
\author[2]{Huiyan Xue\thankmail{Huiyan.Xue@math.uib.no}}
\author[2]{Antonella Zanna\thankmail{Antonella.Zanna@math.uib.no}}
\affil[1]{Department of Mathematics and Mathematical Statistics, Ume\aa{} University, Sweden}
\affil[2]{Department of Mathematics, University of Bergen\\Johannes Brunsgt 12, 5008 Bergen, Norway}
\date{}

\usetikzlibrary{positioning, calc}

\usepackage{float}
\restylefloat{table}

\newcommand*\RR{\mathbb{R}}

\newcommand*\demph[1]{\textbf{#1}}

\newcommand*\inv{^{-1}}

\newcommand*\xm{\xx{-}}
\newcommand*\xp{\xx{+}}
\newcommand*\xz{\xx{\circ}}
\newcommand*\Xm{\XX{-}}
\newcommand*\Xp{\XX{+}}
\newcommand*\Xz{\XX{\circ}}

\newcommand*\XX[1]{X_{#1}}
\newcommand*\xx[1]{x_{#1}}
\newcommand*\YY[1]{Y_{#1}}
\newcommand*\yy[1]{y_{#1}}
\newcommand*\ZZ[1]{Z_{#1}}
\newcommand*\zz[1]{z_{#1}}

\newcommand*\pot{\phi}
\newcommand*\Pot{\Phi}

\newcommand*\perm{\sigma}
\newcommand*\Perm{\Sigma}

\newcommand\arperm{\pi}

\NewDocumentCommand\volperm{O{\perm,\Perm}}{\Psi_{#1}}
\NewDocumentCommand\permact{mm}{(#1) \act #2}
\NewDocumentCommand\volpot{O{\pot,\Pot}O{\sig}}{\mathbf{f}_{#1}\IfNoValueF{#2}{^{#2}}}
\newcommand*\volpota{\mathbf{f}}

\newcommand*\vf{\mathbf{a}}

\newcommand*\pd[2][]{\partial_{#2} #1}

\newcommand*\act{\cdot}

\newcommand*\sig{\pm}

\newcommand*\pfl{p}

\begin{document}

\maketitle
\begin{abstract}
In earlier work, Lomeli and Meiss \cite{lomeli09gff} used a generalization of the symplectic approach to study volume preserving generating differential forms. In particular, for the $\RR^3$ case, the first to differ from the symplectic case, they derived thirty-six one-forms that generate exact volume preserving maps. In \cite{xue2014gf}, Xue and Zanna studied these differential forms in connection with the numerical solution of divergence-free differential equations: can such forms be used to devise new volume preserving integrators or to further understand existing ones? As a partial answer to this question, Xue and Zanna showed how six of the generating volume form were naturally associated to  consistent, first order, volume preserving numerical integrators.
In this paper, we investigate and classify the remaining cases. The main result  is the reduction of the thirty-six cases  to five essentially different cases, up to variable relabeling and adjunction. 
We classify these five cases, identifying two novel classes and associating the other three to volume preserving vector fields under a Hamiltonian or Lagrangian representation. We demonstrate how these generating form lead to consistent volume preserving schemes for volume preserving vector fields in $\RR^3$. 
\end{abstract}

\section{Introduction and background}
\label{laha}
The scope of this paper is the study volume preserving generating forms, with the ultimate goal of exploiting these differential forms to obtain consistent numerical methods that preserve volume for arbitrary volume preserving vector fields. 
This task is particularly hard: there exist \emph{no-go} theorems \cite{chartier2007preserving, iserles07} stating that  it is not possible to construct volume preserving methods for generic $n$-dimensional volume preserving vector fields within the class of B-series methods, a class that includes classical integrators like Taylor-expansion based methods, Runge--Kutta methods and multistep methods. On the other hand, volume preserving methods can be constructed using the technique of \emph{splitting} \cite{McLachlan2002}. Splitting methods correspond to P-series (``P'' for partitioned systems), a generalization of B-series. 
Several splitting techniques can be adopted. The earliest and best known splitting consists in decomposing the vector field in  2D Hamiltonian sub-systems \cite{kang95vpa}, which are then solved by a symplectic method. More recently, research has focussed on explicit splitting methods for classes of vector fields, like polynomial or trigonometric, which are  wide enough to include most interesting cases, but not as large as the space of all possible vector fields \cite{quispel03evp,mclachlan04egi, mclachlan09evp, Xue:2012mf,zanna2014tensor}. 

Differently from the symplectic case, the generating form approach to generate volume preserving numerical methods is not well understood.  
Earlier work by \cite{shangXXgf1, Shang1994} extends the Hamiltonian technique of \cite{F1986, FW1989}, using linear maps in the product space, to volume preserving forms, thus obtaining an equivalent of the Hamilton-Jacobi differential equation \cite{Shang1994}. To obtain a first and second order scheme, Shang had to impose simplifying conditions, requiring the transformation matrix to be a special case of Hadamard matrix. However, the numerical integrators by this approach are quite complicated, as they are defined via implicit maps, although the approach is valid for arbitrary vector fields.
Another generating-functions related approach is due to \cite{GRW199526}: through a special combinations of explicit and implicit maps, Quispel shows that the resulting method is volume preserving. This is a ``correction method'': starting from an arbitrary numerical integrator, one has to consider an extra term (the corrrection) for volume preservation.
The above mentioned two approaches do not use differential forms directly, rather, they use the equivalent condition on the determinant of the Jacobian of the map. For this reason, they use the terminology of ``generating functions'' rather than ``generating forms''.

More recently,  L{\'o}meli and Meiss \cite{lomeli09gff} have studied the problem of volume preserving maps using differential forms and generalization of the symplectic approach. They discussed in detail the $\RR^3$ case, the first to differ from the symplectic case, and described how the generic volume preserving maps can be described by thirty-six one-forms.
That paper paves the background for our investigations. In particular, we are interested in understanding how these differential forms are associated to numerical methods (if any) and whether some of these forms can lead to new techniques to obtain volume preserving maps. These questions were partially addressed in \cite{xue2014gf}, where six of the thirty-six differential one forms were identified and associated to splitting methods. The scope of this paper is to discuss  and classify the remaining cases. 

The main result of the paper is the reduction of the thirty-six cases to five essentially different cases, using equivalence relations (global variable renaming and numerical adjoints). Thereafter, these five cases are classified and three of them associated to known techniques, based on Hamiltonian and/or Lagrangian formalism.  In other words, the generating forms are associated to symplectic splitting methods like Symplectic Euler (SE) for Hamiltonian systems, or Discrete Lagrangian (DL) methods for appropriate Lagrangian functions, or a combination of both. We further identify two  special classes, $S_1$ and $S_2$, that, to our knowledge, do not have a straightforward and direct mechanical interpretation. It is these two special classes that are of particular interest in the search of new volume preserving algorithms. We give explicit formulas of generating one-forms for the class $S_1$ and $S_2$ corresponding to volume preserving first order methods in the specific case of linear vector fields. A general approach is still unknown and will be the subject of future investigation.

\subsection{Background and notation}
\label{sec:background}

We consider a differentiable manifold $\mathcal{M}$. Let $\omega$ be a non degenerate $k$-differential form on $\mathcal{M}$, that is, a $k$-linear map, completely skew-symmetric with respect to its arguments. For each $p \in \mathcal{M}$, $\omega(p): (T_p\mathcal{M})^{\times k} \to \RR$, namely the differential form takes as argument $k$ tangent vectors and returns a number. The coefficients of the differential form might depend on $p\in \mathcal{M}$.
Let $d \omega$ be the $k+1$ form obtained with the usual rules of external derivation. Recall that $\omega$ is \emph{closed}  if $d\omega = 0$ and that $\omega$ is  an exact differential, or simply \emph{exact}, if $\omega = d\nu$, where $\nu$ is a $k-1$ form, called a \emph{primitive}.
By application of Stokes' theorem, $\int_S d \omega = \int_{\partial S} \omega$, valid on any oriented manifold with oriented boundary $\partial S$, to the differential form $d\omega$, it follows that $d^2\omega=0$. Therefore, any exact form, $\omega = d \nu$, is closed, i.e.\ $d\omega = d^2\nu = 0$. The reverse statement is not true in general, but it holds on contractible manifolds, as explained from the following lemma.

\begin{lemma}[Poincar{\'e} lemma]
\label{th:poincare}
A closed form ($d\omega=0$) is locally exact ($\omega=d\nu$), that is, there is a neighborhood $U$ about each point on which $\omega=d\nu$. The statement is globally true on contractible manifolds.
\end{lemma}

In the sequel, we focus on volume forms and their primitives. 
We will assume, otherwise stated, that the differential forms are \emph{non-degenerate}. This means that the coefficients of the form are never simultaneously zero. 

\begin{definition}
A volume form $\Omega$ on a manifold $\mathcal{M}$ is preserved by a $C^{1}$-map $\mathbf{f}: \mathcal{M}\mapsto \mathcal{M}$  if 
\begin{equation}
	\mathbf{f}^{*}\Omega=\Omega,
	\label{eq:vol_pres}
\end{equation}
where $\mathbf{f}^{*}$ denotes the pull-back of $\mathbf{f}$.  The map $\mathbf{f}$ is said to be \emph{canonical} or \emph{volume preserving}.
\end{definition}

In what follows, we let $\mathcal{M}=\RR^n$. Let $\nu$ be any primitive form of $\Omega$, i.e.\ $d\nu=\Omega$. 
Then, condition \eqref{eq:vol_pres} becomes $\mathbf{f}^{*}d\nu-d\nu=0$
and implies $d(\mathbf{f}^{*}\nu-\nu)=0$, hence
$\mathbf{f}^{*}\nu-\nu$ is the exact differential of a $n-2$ form, as
a consequence of Lemma~\ref{th:poincare}. This motivates 
Definitions \ref{def:nu_exact} and \ref{df} below, see \cite{lomeli09gff}. 

\begin{definition}
\label{def:nu_exact}
Let $\nu$ be a primitive of the volume form $\Omega$  and $\mathbf{f}: \RR^n\mapsto \RR^n$ an  exact volume preserving diffeomorphism  such that
\begin{equation}
 \mathbf{f}^{*}{\nu}-\nu=d\lambda
 \label{eq:nu_exact},
\end{equation}
for a $n-2$ form $\lambda$. The differential form $\lambda$ is called a \emph{generating form with respect to $\nu$.}
\end{definition}

Primitives of forms are not uniquely determined: by choosing $\tilde \nu$ another primitive of $\Omega$, the volume preservation condition \eqref{eq:vol_pres} can be written as $\mathbf{f}^{*}d\nu-d\tilde\nu=0$. A procedure similar to the one just described above leads to:
\begin{definition}
\label{df}
Let $\nu, \tilde \nu$ be two primitives of a volume form $\Omega$, i.e.\ $d\nu=d\tilde{\nu}=\Omega$ and   $\mathbf{f}: \RR^n\mapsto \RR^n$ an exact volume preserving diffeomorphism such that
\begin{equation}
 \mathbf{f}^{*}\tilde{\nu}-\nu=d\lambda\label{eq:0806031},
\end{equation}
 for a $n-2$ form $\lambda$.  The $n-2$ differential form $\lambda$ is called a \emph{generating form with respect to $(\nu, \tilde \nu)$.}
\end{definition}

We will consider the choice of canonical coordinates $x_1, \ldots, x_n$ in $\RR^n$ and denote by $\mathbf{x} = (x_1, \ldots, x_n)^T$ the original (old) variables. Given a volume preserving map $\mathbf{f}$, we will denote the transformed (new) variables 
by uppercase letters, i.e.\ $\mathbf{X} = (X_1, \ldots, X_n)^T = \mathbf{f}(\mathbf{x})$.

Volume preservation in $\RR^2$ is equivalent to preservation of area and volume forms are the same as symplectic forms. 
This case is well understood. 
The case $n=3$ is the first case for which volume forms and symplectic forms are different.

\section{Generating forms in $\RR^3$}
\label{volsec}

In \cite{lomeli09gff}, Lomeli and Meiss studied in detail exact volume preserving mappings and generating forms  in $\RR^3$. 
Starting from \eqref{eq:0806031} and using canonical coordinates, they showed that for each choice of primitives $(\nu, \tilde{\nu})$ there are four different generating one-forms, 
\begin{equation}
	\lambda= \phi dx_l + \Phi d X_m,  \qquad \phi \in \{A,B\}, \quad \Phi \in \{C,D\}\qquad l, m \in \{1,2,3\},
	\label{eq:oneform}
\end{equation}
which they identified using four generating functions $A, B, C$ and $D$.

As $\nu$ and $\tilde{\nu}$ can be chosen in three different ways
($x_3dx_1\wedge dx_2$, $x_2dx_3\wedge dx_1$ and $x_1dx_2\wedge dx_3$),
this approach gives a total of thirty-six generating one-forms. Four
of them, corresponding to $\nu= \tilde \nu = x_3 dx_1 \wedge d x_2$,
are shown in Table~\ref{tab:123123}.  Each cell in the table is described by: a generating one-form $\lambda$; two \emph{determining conditions}, determining a lowercase and an uppercase variable;  a \emph{compatibility condition} and two \emph{twist conditions} to guarantee that the three equations are solvable and they give rise to a well-defined volume preserving map.
Altogether, one  there are nine such tables, 
obtained by even permutations of the $x_1,x_2,x_3$ and the $X_1,X_2,X_3$ variables.

\begin{table}[t]
\centering
  \begin{tabular}{ l | c | c | }
  $X_3 dX_1 \wedge dX_2$ $-$ && \\
  \quad $x_3 dx_1 \wedge dx_2$ & $A dx_1$& $B dx_2$\\  \hline
    $C dX_1$ & $\begin{array}{ccc}
 \lambda = A(x_1, x_2, X_1) dx_1 \\
 \qquad \quad \mbox{}+ C(x_1, X_1, X_2) d X_1\\[5pt]
x_3=\partial_{x_2}A \\
\partial_{X_1} A = \partial_{x_1} C\\
X_3 = -\partial_{X_2} C\\
\frac{\partial X_1}{\partial x_3} \not=0, \quad \frac{\partial x_1}{\partial X_3} \not=0
\end{array} $ &  $\begin{array}{ccc}
\lambda = B(x_1, x_2, X_1) dx_2 \\
\qquad \quad \mbox{}+ C(x_2, X_1, X_2) d X_1\\[5pt]
x_3=-\partial_{x_1}B \\
\partial_{X_1} B = \partial_{x_2} C\\
X_3 = -\partial_{X_2} C\\
\frac{\partial X_1}{\partial x_3} \not=0, \quad \frac{\partial x_2}{\partial X_3} \not=0
\end{array} $\\ \hline
     $D d X_2$ & $\begin{array}{ccc}
\lambda = A(x_1, x_2, X_2) dx_1 \\
\qquad\quad  \mbox{}+ D(x_1, X_1, X_2) d X_2\\[5pt]
x_3=\partial_{x_2}A \\
\partial_{X_2} A = \partial_{x_1} D\\
X_3 = \partial_{X_1} D\\
\frac{\partial X_2}{\partial x_3} \not=0, \quad \frac{\partial x_1}{\partial X_3} \not=0
\end{array} $ &  $\begin{array}{ccc}
\lambda = B(x_1, x_2, X_2) dx_2 \\
 \qquad \quad \mbox{}+ D(x_2, X_1, X_2) d X_2\\[5pt]
x_3=-\partial_{x_1}B \\
\partial_{X_2} B = \partial_{x_2} D\\
X_3 = \partial_{X_1} D\\
\frac{\partial X_2}{\partial x_3} \not=0, \quad \frac{\partial x_2}{\partial X_3} \not=0
\end{array} $\\
    \hline
  \end{tabular}
  \caption{The four basic types of generating 1-forms $\lambda$ for $\nu = \tilde \nu = x_3 dx_1 \wedge d x_2$, adapted from \cite{lomeli09gff}. All the other tables are obtained by applying cyclic even permutations to the variables $(x_1,x_2,x_3)$ in $\nu$ and $(X_1,X_2,X_3)$ in $\mathbf{f}^*\tilde{\nu}$.}
  \label{tab:123123}
\end{table}

In \cite{xue2014gf}, Xue and Zanna studied these generating forms with the goal of associating them to volume preserving vector fields and numerical volume preserving  integrators. They succeeded in identifying six cases and associating them to splitting methods 
in using two potential functions: each of the functions gave rise to a two-dimensional Hamiltonian approximated by a Symplectic Euler (SE) method. The six cases are not fundamentally different. First of all, having chosen one case, two of the other cases correspond to a global variable renaming (say, $(x_1,x_2,x_3) \to (x_2,x_3,x_1)$ and $(X_1,X_2,X_3) \to (X_2,X_3,X_1)$ simultaneously). Relabeling the variables in a numerical method does not give a new numerical method.
Secondly, the remaining three cases were obtained by exchanging lower
cases  and reversing time, in other words, they corresponded to the
\emph{adjoint} numerical methods of the previous three cases. 
As it is known how to obtain the adjoint of a given method \cite{hairer87sod}, these cases are not interesting per se either.
All this indicates that there is redundancy in the  thirty-six
cases. To classify and understand which one forms are related to known
methods and which forms can lead to genuinely new approaches, we need
to establish equivalence classes, so that our search can be restricted
to a single differential form for each class.

\section{Problem statement}

We consider the ordinary differential equation
\begin{equation}
\dot{\mathbf{x}}=\mathbf{a}(\mathbf{x}),  ~~\mathbf{x}(0)=\mathbf{x}_{0}, 
\label{eq:vf}
\end{equation}
where $\mathbf{x} \in \mathcal{M}$ and $\mathbf{a} :\mathcal{M} \to
T_\mathbf{x} \mathcal{M}$, $\mathbf{a}(\mathbf{x}) = [a_1(\mathbf{x}),
\ldots, a_n(\mathbf{x})]^T$, is a smooth vector field. We denote by
$\mathbf{a}^t$ the flow of \eqref{eq:vf} and by $\omega$ a $k$-form. 

Derivatives of differential forms along the flow are called \emph{Lie derivative} and defined as
\begin{equation}
	L_{\mathbf{a}} \omega = \frac{d}{dt} (\mathbf{a}^t)^* \omega \Big|_{t=0}
\label{eq:lie_deriv}
\end{equation}
Let $\Omega$ be a volume form on $\mathcal{M}$. We say that the vector field $\mathbf{a}$ is \emph{volume preserving} if
\begin{equation}
	L_{\mathbf{a}} \Omega =0.
	\label{eq:volpres_a}
\end{equation}
The flow $\mathbf{a}^t$ is volume preserving if its vector field $\mathbf{a}$ is volume preserving. By \eqref{eq:lie_deriv}, this implies that
\begin{displaymath}
 	(\mathbf{a}^t)^* \Omega = \Omega.
\end{displaymath}
Let $i_\mathbf{a}$ a contraction, that is, for any $k$-form $\omega$, $i_\mathbf{a} \omega$ is the $k-1$ form $\omega (\mathbf{a}, \mbox{} \cdot \mbox{})$ obtained by inserting $\mathbf{a}$ in the first slot.
By Cartan's formula for Lie derivatives,
\begin{equation}
	L_{\mathbf{a}} \Omega = d (i_\mathbf{a} \Omega) + i_{\mathbf{a}} d \Omega,
\label{eq:Cartan}
\end{equation}
it follows that a vector field $\mathbf{a}$ is volume preserving if $d (i_\mathbf{a} \Omega) =0$, that is, $ i_\mathbf{a} \Omega $ is closed (as $d \Omega =0$, being $\Omega$ a $n$-form).

\begin{definition}
Let $\lambda$ be a $n-2$ form and $\Omega$ a volume form on $\mathcal{M}$. A vector field $\mathbf{a}$ on $\mathcal{M}$  is \emph{exact volume preserving} with respect to the potential form $\lambda$ if
\begin{equation}
	i_\mathbf{a} \Omega = d \lambda.
	\label{eq:exactvp}
\end{equation}
\end{definition}
In particular, it follows from Poincar{\'e}'s lemma~\ref{th:poincare} that, when $\mathcal{M}=\RR^n$, globally defined volume preserving vector fields are also exact.

Let us consider  the case $n=3$ in more detail. For any vector $\mathbf{v}_i \in \RR^3$, $i=1,2,3$, we have $\Omega(\mathbf{v}_1, \mathbf{v}_2,\mathbf{v}_3)= dx_1 \wedge dx_2 \wedge dx_3 (\mathbf{v}_1, \mathbf{v}_2,\mathbf{v}_3)= \det[\mathbf{v}_1, \mathbf{v}_2,\mathbf{v}_3]$. 
By direct computation, we see that $\Omega (\mathbf{a}, \mathbf{v}, \mathbf{w}) = [a_1 dx_2 \wedge dx_3 + a_2 dx_3 \wedge d x_1 + a_3 d x_1 \wedge d x_2] (\mathbf{v}, \mathbf{w})$ for any $\mathbf{v}, \mathbf{w}$. We deduce that
\begin{displaymath}
 	i_\mathbf{a}\Omega = a_1 dx_2 \wedge dx_3 + a_2 dx_3 \wedge d x_1 + a_3 d x_1 \wedge d x_2.
\end{displaymath}
There are three natural choices of the one-form $\lambda$, i.e.\ $\lambda_i = F^{i}(x_1,x_2,x_3) d x_i$, $i=1,2,3$, where the $F^i$s are arbitrary function. Then,
\begin{displaymath}
	d \lambda_i = \sum_{j=1}^3 \partial_{x_j} F^i d x_j \wedge d x_i, \qquad i=1,2,3,
\end{displaymath}
and, by \eqref{eq:exactvp}, we deduce that a volume preserving vector field, exact with respect to the form $\lambda_i$, must have the following form:
\begin{align}
\label{eq:vf1}
 F^1 d x_1 &: \qquad
 \begin{array}{ccc}
a_1 &=& 0,\\
a_2 &=& \partial_{x_3} F^{1},\\
a_3 &=& -\partial_{x_2} F^1,
\end{array} \\
\label{eq:vf2}
 F^2 d x_2 &: \qquad
\begin{array}{ccc}
 a_1 &=& -\partial_{x_3} F^2,\\
a_2 &=& 0,\\
a_3 &=& \partial_{x_1} F^2,
\end{array}\\
\label{eq:vf3}
F^3 d x_3 &:\qquad
\begin{array}{ccc}
a_1 &=& \partial_{x_2} F^3,\\
a_2 &=& -\partial_{x_1} F^3,\\
a_3 &=& 0.\end{array}
\end{align}
A generic three-dimensional volume preserving vector field will be a
combination of \eqref{eq:vf1}-\eqref{eq:vf3} above. Note that only two
of them are linearly independent, the same yields for the choices of
$\lambda_i$: for instance,  as long as $F^1, F^2$ depend on all
variables, $d \lambda_1$ and $d \lambda_2$ generate
all the $dx_i \wedge dx_j$, $i, j=1,2,3$ although $\lambda_1$ contains
only $d x_1$ and $\lambda_2$ only $dx_2$. 

This is a consequence of a well known result for volume preserving flows in $\RR^n$: any $n$-dimensional volume preserving differential equation is described by $n-1$ independent potential functions (see for instance \cite{McLachlan2002,kang95vpa, FW1989,xue2014gf} and discussion therein). One of the earliest normalization of the $n-1$ independent potential functions is due to Weyl \cite{weyl40tmo}.

It is natural to draw a connection between the generating forms
\eqref{eq:oneform}
 and the volume preserving vector field $\vf$. In particular, we address the following questions:
\begin{itemize}
\item How do the generating functions $\phi, \Phi$ relate to  the
  vector field $\vf$ and the $F^{i}$s, if there is any relation?
\item Can the generating forms \eqref{eq:oneform} be naturally
  associated to volume preserving numerical methods for \eqref{eq:vf}
  whenever the vector field $\mathbf{a}$ in \eqref{eq:vf} is volume
  preserving? 
\end{itemize}
An important property of numerical methods is
\emph{consistency}. Numerical methods introduce a discrete time step
$h$. A fundamental property required to the method is that, in the limit $h\to 0$,
$\frac{\mathbf{X}-\mathbf{x}}{h} = \vf$, that is, the method solves the given differential equation.

We now define what is the main goal of this paper: to find suitable
potential functions $\pot$ and $\Pot$ for a divergence free vector
field $\vf$ which also are a solution of the \emph{consistency
  problem}, below rephrased in the formalism of this paper for convenience.

The main goal of this section is to show that, subject to consistency,
there are only five different classes of generating volume forms in $\RR^3$.

\subsection{Defining equations and compatibility conditions}

Consider two arbitrary smooth functions
\begin{equation}
\pot,\Pot \colon \RR^3 \to \RR
,
\end{equation}
and an arbitrary sign, which will be denoted by $\sig$ in the equations.

Define the corresponding mapping $\volpot\colon (\yy1,\yy2,\yy3) \mapsto(\YY1,\YY2,\YY3)$ implicitly by the following equations.
\begin{subequations}
	\label{eq:volpot}
	\begin{align}
		\partial_{2} \pot(\YY3, \yy2, \yy3) &=\yy1 \label{eq:volpot1}  \\
		\partial_{3} \Pot(\YY3,\YY2,\yy3) &\sig \partial_{1} \pot(\YY3,\yy2, \yy3) = 0 \label{eq:volpot2}\\
		\YY1 &= \partial_{2} \Pot(\YY3, \YY2, \yy3). \label{eq:volpot3}
\end{align}
\end{subequations}

The mapping $\volpot$ is well defined as soon as 
$\pd[\Pot]{32} \neq 0$ and $\pd[\pot]{21} \neq 0$.

We define the \demph{action} of a permutation $\perm\colon\set{1,2,3}\to\set{1,2,3}$ on an element $\mathbf{x}\in\RR^3$ by
\begin{align}
	(x \act \perm)_i \coloneqq x_{\perm(i)}
	.
\end{align}
Consider two permutations $\perm$ and $\Perm$:
\begin{equation}
	\perm,\Perm \colon \set{1,2,3} \to \set{1,2,3}	.
\end{equation}
For any map $\volpota \colon \RR^3 \to \RR^3$, we define the map $\permact{\perm,\Perm}{\volpota}$ by
\begin{align}
	\paren[\big]{\permact{\perm,\Perm}{\volpota}}(x) \coloneqq \volpota(x \act \perm) \act \Perm\inv
	.
\end{align}

\subsection{The consistency problem}
\begin{definition}
	\label{def:conssol}
We say that a pair of maps
$\pot, \Pot \colon \mathfrak{X}_0(\RR^3) \to \mathcal{C}^{\infty}(\RR,\RR^3)$,
where $\mathfrak{X}_0(\RR^3)$ denotes the space of divergence-free vector fields on $\RR^3$,
is a \demph{solution to the consistency problem} for the pair of permutations $(\perm,\Perm)$ if
 the map $h \mapsto \permact{\perm,\Perm}{\volpot[\pot(h\vf),\Pot(h\vf)]}$ is \emph{consistent} with $\vf$.

This means that
\begin{align}
	\lim_{h\to 0} \frac{1}{h}\paren[\Big]{\paren[\big]{\permact{\perm,\Perm}{\volpot[\pot(h\vf),\Pot(h\vf)]}}(\mathbf{x}) - \mathbf{x}} = \vf
	.
\end{align}
\end{definition}

Note that for any fixed vector field $\vf$, the element $\pot(\vf)$ is a potential, i.e., $\pot(\vf)$ is itself a function from $\RR^3$ to $\RR$.
Note also that a method obtained from \autoref{def:conssol}  will automatically have order one.


\subsection{Case Reduction}

Note that \emph{all the defining equations} and \emph{compatibility
  conditions} corresponding to \eqref{eq:oneform} are of the form
\eqref{eq:volpot1}--\eqref{eq:volpot3} for some choices of $\perm,\Perm, \pot,
\Pot$.

As the permutations $\perm$ and $\Perm$ each range over $3! = 6$
values, it would seem that there are $6\times 6 = 36$ different
problems \eqref{eq:volpot1}--\eqref{eq:volpot3}. 

The purpose of this section is to explain that there are in fact
\emph{at most five} different solutions to the consistency problem, up
to equivalence relations (relabeling and adjunction). 
What this means is that the problem of finding potentials for the thirty-six permutation cases reduces to finding the solution of only five cases.

\subsubsection{Reduction by relabeling}

\begin{proposition}
	\label{prop:permactalg}
	Consider an arbitrary permutation $\arperm$ of $\set{1,2,3}$.
	We have
	\begin{align}
		\permact{\arperm,\arperm}{\paren[\Big]{\permact{\perm,\Perm}{\volpot}}} = \permact{\arperm \perm,\arperm\Perm}\volpot
		.
	\end{align}
\end{proposition}
\begin{proof}
	Observe that for any function $\volpota\colon\RR^3\to\RR^3$ and any permutations $\arperm$ and $\arperm'$ we have
	\(
		\permact{\arperm,\arperm'}{\paren[\big]{\permact{\perm,\Perm}{\volpota}}}
		=
		{\permact{\arperm\perm,\arperm'\Perm}{\volpota}}
	\).
	The claim follows immediately.
\end{proof}

\begin{proposition}
	\label{prop:trivperm}
	Consider an arbitrary permutation $\arperm$ of $\set{1,2,3}$.
	If $\pot, \Pot$ is a solution for 
         $(\perm,\Perm)$ (\autoref{def:conssol}),
        then the pair of maps
        \begin{align}
		\widetilde{\pot}(\vf) \coloneqq \pot\paren[\big]{\permact{\arperm\inv,\arperm\inv}{\vf}}		,
		\qquad
		\widetilde{\Pot}(\vf) \coloneqq \Pot\paren[\big]{\permact{\arperm\inv,\arperm\inv}{\vf}}		
	\end{align}
		is a solution for
        $(\arperm\perm,\arperm\Perm)$.
\end{proposition}
\begin{proof}
		As $\pot,\Pot$ is a solution for $(\perm,\Perm)$, it means that $\volpot[\widetilde{\pot}(h\vf),\widetilde{\Pot}(h\vf)]$ is consistent with the vector field $\permact{\arperm\inv,\arperm\inv}{\vf}$.
		Using \autoref{prop:permactalg}, we obtain that the map $\permact{\arperm\perm,\arperm\Perm}{\volpot[\pot(h\vf),\Pot(h\vf)]}$ is consistent with the vector field $\permact{\arperm,\arperm}{\paren[\big]{\permact{\arperm\inv,\arperm\inv}{\vf}}}$.
		We conclude by using $\permact{\arperm,\arperm}{\paren[\big]{\permact{\arperm\inv,\arperm\inv}{\vf}}} = \vf$.
\end{proof}

\subsubsection{Reduction by adjunction}
For a map $\pot \colon \RR^3 \to \RR$, we define the permuted map $\pot\act\pfl$ by
\begin{align}
	\paren{\pot\act\pfl}(y) \coloneqq \pot(y\act p)
	,
\end{align}
where $\pfl$ is the special permutation defined by:
\begin{align}
	\label{eq:defpfl}
	\pfl \coloneqq (1,2,3) \mapsto (3,2,1)
	.
\end{align}

In concrete terms it just means that the action of $p$ switches the first and last argument:
\begin{align}
	\paren{\pot\act\pfl}(\yy1,\yy2,\yy3) = \pot(\yy3,\yy2,\yy1)
\end{align}

We spend most of the time in this section showing that we can easily compute the inverse of maps such as $\permact{\perm,\Perm}{\volpota}$ and $\volpot$.

\begin{proposition}
	\label{prop:volpotinv}
	The map ${{\volpot[\Pot\act\pfl,\pot\act\pfl]}}$ is the inverse of $\volpot$.
\end{proposition}
\begin{proof}
Let us pose $Z \coloneqq \volpot[\Pot\act\pfl,\pot\act\pfl](z)$.
	We also define for convenience
	\begin{align}
		\label{eq:zyswitch}
		y \coloneqq Z
		,
		\qquad
		Y \coloneqq z
		.
	\end{align}
The aim is to prove that $Y = \volpot(y)$, which will finish the proof.

We now follow the definition \eqref{eq:volpot}.
From \eqref{eq:volpot1} we obtain
\begin{align}
	\partial_2(\Pot\act\pfl)(\ZZ3,\zz2,\zz3) = \zz1
\end{align}
which gives
\begin{align}
	\partial_2\Pot(\zz3,\zz2,\ZZ3) = \zz1
\end{align}
which, using \eqref{eq:zyswitch} gives
\begin{align}
	\partial_2\Pot(\YY3,\YY2,\yy3) = \YY1
\end{align}
which is exactly \eqref{eq:volpot3}.

The same computation shows that \eqref{eq:volpot3} transforms into \eqref{eq:volpot1}.

As to \eqref{eq:volpot2}, we have
\begin{align}
	\partial_3(\pot\act\pfl)(\ZZ3,\ZZ2,\zz3) \sig \partial_1(\Pot\act\pfl)(\ZZ3,\zz2,\zz3) = 0
\end{align}
which, gives
\begin{align}
	\partial_1\pot(\zz3,\ZZ2,\ZZ3)  \sig \partial_3\Pot(\zz3,\zz2,\ZZ3) = 0
\end{align}
and using \eqref{eq:zyswitch}:
\begin{align}
	\partial_1\pot(\YY3,\yy2,\yy3)  \sig \partial_3\Pot(\YY3,\YY2,\yy3) = 0
\end{align}
which is exactly \eqref{eq:volpot2}.

We conclude that $Y = \volpot(y)$.
\end{proof}

\begin{proposition}
	\label{prop:permactinv}
	Given any invertible map $\volpota$, the map $\permact{\Perm,\perm}{\volpota\inv}$ is the inverse of $\permact{\perm,\Perm}{\volpota}$.
\end{proposition}
\begin{proof}
	Let us define $X$ and $x$ by:
	\begin{align}
		X = \paren[\big]{\permact{\perm,\Perm}{\volpota}}(x) = \volpota(x\act\perm)\act\Perm\inv
	\end{align}
	We then have
	\(
	X\act\Perm = \volpota(x\act\perm)
	\),
	so
	\(
	x\act\perm = \volpota\inv(X\act\Perm)
	\),
	and
	\begin{align}
		x = \volpota\inv(X\act\Perm)\act\perm\inv = \paren[\big]{\permact{\Perm,\perm}{\volpota\inv}}(X)
	,
	\end{align}
	which concludes the proof.
\end{proof}

\begin{proposition}
	\label{prop:permactvolpotinv}
	The map $\permact{\Perm ,\perm}{\volpot[\Pot\act\pfl,\pot\act\pfl]}$ is the inverse of $\permact{\perm,\Perm}{\volpot[\pot,\Pot]}$.
\end{proposition}
\begin{proof}
	Immediate consequence of \autoref{prop:volpotinv} and \autoref{prop:permactinv}.
\end{proof}

We thus obtain the following
\begin{proposition}
	\label{prop:switchperm}
	If $(\perm,\Perm)$ has a solution $\pot(\vf),\Pot(\vf)$ (\autoref{def:conssol}), then $(\Perm ,\perm)$ has the solution given by the potentials
	\begin{align}
		\widetilde\pot(\vf) \coloneqq \Pot(-\vf)\act\pfl 
		\qquad
		\widetilde\Pot(\vf) \coloneqq \pot(-\vf)\act\pfl
		.
	\end{align}
\end{proposition}

\begin{proof}
	By \autoref{prop:permactvolpotinv}, $\permact{\Perm,\perm}{\volpot[\widetilde{\pot}(h\vf),\widetilde{\Pot}(h\vf)]}$ is the inverse of $\permact{\perm,\Perm}{\volpot[\pot(-h\vf),\Pot(-h\vf)]}$, which by assumption is consistent with $-\vf$.
\end{proof}

\subsection{Main result}
We are now ready to introduce the main result of this section.

\begin{theorem}
	\label{prop:fivecases}
	If one has a solution $(1,\tau)$ (\autoref{def:conssol}), where $\tau$ is one of the \emph{five} permutations consisting of the identity, the three mirror symmetries and one rotation (see \autoref{fig:triangle}), then one has a solution for all permutations $(\perm,\Perm)$.
\end{theorem}
\begin{proof}
	First, suppose that $\tau$ is one of the permutations for which we have a solution, i.e., we have a solution for $(1,\tau)$.
	Then by \autoref{prop:switchperm}, we obtain a solution for $(\tau,1)$, and using \autoref{prop:trivperm}, we obtain a solution for $(1,\tau\inv)$.
	As a result, we obtain a solution for all the six cases of the form $(1,\tau)$, for any permutation $\tau$.
	Suppose now that we have a pair of permutations $(\perm,\Perm)$.
	By \autoref{prop:trivperm}, as we have a solution for $(1,\perm\inv\Perm)$, we have a solution for $(\perm,\Perm)$.
\end{proof}

In order to write the equations for the map $X = \permact{\perm,\Perm}{\volpot}(x)$ we will use the following notations, for fixed permutations $\perm$ and $\Perm$.
We write
\begin{subequations}
	\label{eq:notpzm}
\begin{align}
\xp = \yy{\perm\inv(1)},\xz = \yy{\perm\inv(2)},\xm = \yy{\perm\inv(3)}
\end{align}
Similarly, we write
\begin{align}
	\Xp = \YY{\Perm\inv(1)}, \Xz = \YY{\Perm\inv(2)}, \Xm = \YY{\Perm\inv(3)}
\end{align}
\end{subequations}

The general equations for $X = \permact{\perm,\Perm}{\volpot}(x)$ are written as
\begin{subequations}
	\label{eq:-o+general}
	\begin{align}
		\partial_{\xz} \pot(\Xm, \xz, \xm) &=\xp  \\
		\partial_{\xm} \Pot(\Xm,\Xz,\xm) & \sig \partial_{\Xm} \pot(\Xm,\xz, \xm) = 0 \\
		\Xp &= \partial_{\Xz} \Pot(\Xm, \Xz, \xm)
\end{align}
\end{subequations}


The meaning of \autoref{prop:fivecases} is the following.
By \autoref{prop:trivperm}, the types of equation are classified by the permutation $\tau \coloneqq \perm\inv \circ \Perm$.
Moreover, by \autoref{prop:switchperm}, if we have a solution for the permutation $\tau$, we have a solution for the permutation $\tau\inv$.
Now, the permutation group, depicted on \autoref{fig:triangle}, consists of the identity, three mirror symmetries, and two rotations.
The identity and the mirror symmetries are their own inverse, so \autoref{prop:switchperm} is trivial in those case.
However, we see that both rotations are the inverse of one another,
and it thus suffices to solve the problem corresponding to one
rotation.

%
%

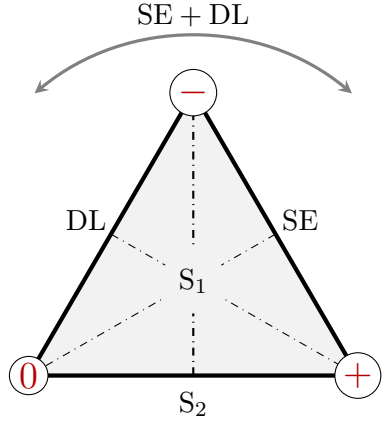
\begin{figure}
	\centering
\begin{tikzpicture}[
	pzm/.style={font=\Large\bfseries, inner sep=1pt, color=red!70!black,  shape=circle, fill=white, draw=black },
	mirror/.style={dashdotted},
	rotation/.style={>=stealth, ->, very thick, gray},
	birotation/.style={rotation, <->,},
scale=2.5]
 
\begin{scope}[rotate=210]
    \coordinate (Z) at (0:1);
    \coordinate  (P) at (120:1);
    \coordinate (M) at (-120:1);
\end{scope}

    \draw [ultra thick, fill=gray!10] (P) -- (M) -- (Z) -- cycle;

	\coordinate (z) at ($ (P)!.5!(M) $);
	\coordinate (m) at ($ (P)!.5!(Z) $);
	\coordinate (p) at ($ (M)!.5!(Z) $);

	\newcommand*\mirrorsep{1.1}
	\node at ($(Z)!\mirrorsep!(z)$) {$\mathrm{SE}$};
	\node at ($(M)!\mirrorsep!(m)$) {$\mathrm{S}_2$};
	\node at ($(P)!\mirrorsep!(p)$) {$\mathrm{DL}$};

	\draw[mirror] (P) -- (p);
	\draw[mirror, thick] (M) -- (m);
	\draw[mirror] (Z) -- (z);

	\node[pzm,] at (P) {$+$};
	\node[pzm,] at (M) {$-$};
	\node[pzm,] at (Z) {$0$};

	\node[shape=circle,  text opacity=1, fill=gray!10] at (0,0) {$\mathrm{S}_1$};



	\newcommand*\rotdist{1.3}
	\newcommand*\rotangle{40}
 
	\begin{scope}[rotate=90]
	\draw[birotation] (-\rotangle:\rotdist) arc[start angle=-\rotangle, end angle=\rotangle, radius=\rotdist];
	\node[above] at(0:\rotdist) {$\mathrm{SE+DL}$};
	\end{scope}
 
\end{tikzpicture}
\caption{A picture of the permutation group of three elements.
It consists of the identity $\mathrm{S}_1$, the mirror symmetries called $\mathrm{S}_2$, $\mathrm{DL}$, and $\mathrm{SE}$. 
The name of the two last mirror symmetries is explained in \autoref{sec:SE} and \autoref{type4}, respectively.
Finally, the left and right rotations, which reduce to only one problem, are denoted here by $\mathrm{SE+DL}$, for reasons explained in \autoref{t2}.}
\label{fig:triangle}
\end{figure}

\section{A classification and description of the five generating volume forms in $\RR^3$}
Having proved that there are only five classes of generating volume
forms in $\RR^3$ (up to relabeling and adjunction) for a given vector
field $\vf$ in $\RR^3$, we proceed with a classification.

\subsection{Class SE+SE}
\label{sec:SE}

We commence with case SE+SE. Its generating differential form was
already discussed and extended to the $n$-dimensional case in
\cite{lomeli09gff}. This case was also discussed in detail in
\cite{xue2014gf}, where it was associated to numerical integrators
consisting of combinations of Symplectic Euler methods. 

With the notation of this paper, we have $\Sigma(+)=1, \Sigma(\circ)=2,\Sigma(-)=3$ and  $\sigma(-)=1,\sigma(\circ)=2,\sigma(+)=3$, see Figure~\ref{fig:triangle}. From $\tau=\Sigma^{-1}\circ\sigma$, we see that $\mathrm{sign}(\tau)=-1$, therefore (\ref{eq:-o+general}) becomes
 \begin{align}
  x_3&=\partial_{x_2}\phi(x_1,x_2,X_3),\\
 \partial_{X_3}\phi(x_1,x_2,X_3)&=\partial_{x_1}\Phi (x_1,X_2,X_3),\\
 X_1&=\partial_{X_2}\Phi (x_1,X_2,X_3),
 \label{eq:case5}
 \end{align}
with twist conditions
\begin{equation*}
\frac{\partial X_1}{\partial x_1}\neq 0,~~\frac{\partial x_3}{\partial X_3}\neq 0,
\end{equation*}
and generating form
\begin{displaymath}
	\lambda=\phi (x_1,x_2,X_3) dx_1+ \Phi (x_1,X_2,X_3)  dX_3.
\end{displaymath}
As $\lambda$ is combination of differentials $dx_1$ and $d X_3$, it is natural to consider vector fields generated by \eqref{eq:vf1} and \eqref{eq:vf3}.
Setting $\Phi =x_1X_2+\Delta t F^{3}(x_1, X_2,X_3)$ and $\phi=x_2X_3+\Delta tF^{1}(x_1,x_2,X_3)$, \eqref{eq:case5} leads to the first order volume preserving scheme 
\begin{align}
X_1&=x_1+\Delta t \partial_{X_2}F^{3}(x_1,X_2,X_3),\\
X_2&=x_2-\Delta t \partial_{x_1}F^{3}(x_1,X_2,X_3)+\Delta t\partial_{X_3}F^{1}(x_1,x_2,X_3),\\
X_3&=x_3-\Delta t \partial_{x_2}F^{1}(x_1,x_2,X_3),
\end{align}
which is equivalent to two steps of the Symplectic Euler (SE) method to solve each of the two 2D Hamiltonian systems \eqref{eq:vf3} and \eqref{eq:vf1}.

\subsection{Class DL+SE}
\label{t2}

We have $\Sigma(+)=1, \Sigma(-)=2,\Sigma(\circ)=3$, and, as for all cases under consideration in this paper,  $\sigma(-)=1,\sigma(\circ)=2,\sigma(+)=3$. We have $\mathrm{sign}(\tau)=1$ and (\ref{eq:-o+general}) becomes
 \begin{align}
	 x_3&=\partial_{x_2}\phi(x_1,x_2,X_2),\\
  	\partial_{X_2}\phi(x_1,x_2,X_2)&=\partial_{x_1}\Phi (x_1,X_2,X_3),\\
 	X_1&=-\partial_{X_3}\Phi (x_1,X_2,X_3),
 \label{eq:case4}
 \end{align}
with twist conditions
\begin{equation*}
	\frac{\partial x_1}{\partial X_1}\neq 0,~~\frac{\partial X_2}{\partial x_3}\neq 0,
\end{equation*}
and generating form
\begin{displaymath}
	\lambda=\phi(x_1,x_2,X_2) dx_1+  \Phi (x_1,X_2,X_3) dX_2.
\end{displaymath}
As $\lambda$ is combination of differentials $dx_1$ and $d X_2$, it is natural to consider vector fields generated by \eqref{eq:vf1} and \eqref{eq:vf2}. 

Note that $\phi$ is a function of both $x_2$ and $X_2$ and that $x_3$ is determined by $x_2, X_2$. 
This points to an interpretation of the Hamiltonian system defined by $H(x_2,x_3)=F^1(x_1, x_2, x_3) $, where $x_2 \equiv q$, $x_3\equiv p$, and $x_1$ is treated as a constant,
\begin{align}
 \dot{x}_1&=0,\\
 \dot{x}_2&=\partial_{x_3}F^{1}(x_1,x_2,x_3),\\
 \dot{x}_3&=-\partial_{x_2}F^{1}(x_1,x_2,x_3).
 \end{align}
 by a \emph{Lagrangian} formulation, with Lagrangian function
\begin{displaymath}
 	L^1(x_1, x_2, \dot x_2) = x_3 \dot x_2 - H = x_3 \dot x_2 - F^1(x_1, x_2,x_3)
\end{displaymath}
\cite{goldstein01cm}.
Consider a \emph{discrete Lagrangian} $ L^1_d= \Delta t L^1(x_1, x_2, (X_2-x_2)/\Delta t)$ \cite{marsden01}. 
With the choice
\begin{displaymath}
	\phi(x_1, x_2, X_2) =L_d^1 (x_1, x_2, X_2),
\end{displaymath}
we see that the first equation of \eqref{eq:case4} is satisfied, and, moreover, an intermediate variable for $x_3$ is obtained,
\begin{displaymath}
	\tilde x_3 = - \partial_{X_2} \phi (x_1, x_2, X_2).
\end{displaymath}

For the $\Phi$ function, choose
\begin{displaymath}
  	\Phi(x_1, X_2, X_3) = -x_1X_3 + \Delta t F^2(x_1, X_2, X_3).
\end{displaymath}
The third equation of \eqref{eq:case4} gives $X_1 = x_1 - \Delta t \partial_{X_3} F^2(x_1,X_2,X_3)$ and $ \Phi_{x_1} = - X_3 + \Delta t \partial_{x_1}F^2(x_1,X_2,X_3)$.

Altogether, we obtain
\begin{align}
	x_3 &= \partial_{x_2}  L_d^1 (x_1, x_2, X_2) \\
	\tilde x_3 &= - \partial_{X_2}  L_d^1 (x_1, x_2, X_2) \qquad (= - \partial_{X_2} \phi) \\
	X_3 &= \tilde x_3 + \Delta t \partial_{x_1}F^2(x_1,X_2,X_3)\qquad ( \hbox{compat.\ cond.\ } \partial_{X_2} \phi = \partial_{x_1} \Phi) \\
	X_1 &= x_1 - \Delta t \partial_{X_3} F^2(x_1,X_2,X_3)
\end{align}
which is a combination of a Discrete Lagrangian method (DL) for \eqref{eq:vf1} and a Symplectic Euler (SE) for \eqref{eq:vf2}. As long as the discrete Lagrangian function $L_d^1$ is a consistent approximation to the continuous one, the composed method has at least order one.

\subsection{Class DL+DL}
\label{type4}

We have $\Sigma(-) = 2$, $\Sigma(\circ) = 1$, $\Sigma(+) = 3$  and $\mathrm{sign}(\tau) = -1$ and (\ref{eq:-o+general}) becomes
\begin{align}
	x_3 &= \partial_{x_2} \phi(x_1, x_2, X_2), \\
	\partial_{X_2} \phi(x_1, x_2, X_2) &= \partial_{x_1} \Phi(x_1, X_2, X_1), \\
	X_3 &= \partial_{X_1} \Phi(x_1, X_2, X_1),
\label{eq:case3}
\end{align}
with twist conditions
\begin{displaymath}
	\frac{\partial x_1}{\partial X_3} \not=0, \quad \frac{\partial X_2}{\partial x_3} \not=0,
\end{displaymath}
and generating form 
\begin{displaymath}
	\lambda=\phi(x_1,x_2,X_2)  dx_1+\Phi (x_1,X_1,X_2) dX_2.
\end{displaymath}
The generating form indicates that one should look for vector fields of the form \eqref{eq:vf1} and \eqref{eq:vf2}. 

Similarly to the procedure described above, we choose $\phi = L_d^1$, generating the intermediate approximation $(x_1, X_2,\tilde x_3)$ (recall that $x_1$ is kept constant).
Also the system \eqref{eq:vf2} is interpreted as a Lagrangian system, with Lagrangian function $L^2(x_1, \dot x_1, X_2)=  \tilde x_3 \dot x_1 - F^2(x_1,X_2, \tilde x_3)$. Now, $X_2$ is kept constant. We set 
\begin{displaymath} 
	\Phi(x_1, X_1, X_2) = -L_d^2 (x_1,X_1,X_2), 
\end{displaymath}	
where  $L_d^2 (x_1,X_1,X_2)= \Delta t L^2(x_1, (X_1-x_1)/\Delta t, X_2)$ is a discrete Lagrangian approximation to $L^2$.

Altogether, we obtain
\begin{align}
	x_3 &= \partial_{x_2}  L_d^1 (x_1, x_2, X_2), \\
	\tilde x_3 &= - \partial_{X_2}  L_d^1 (x_1, x_2, X_2)  \qquad (= - \partial_{X_2} \phi) \\
	\tilde x_3 &=   \partial_{x_1} L_d^2 (x_1, X_1, X_2)  \qquad (\hbox{compat.\ cond.\ } \partial_{X_2} \phi = \partial_{x_1} \Phi)\\
	X_3 &=  -\partial_{X_1} L_d^2(x_1,X_1,X_2),
\end{align}
which is a combination of a Discrete Lagrangian methods (DL) for \eqref{eq:vf1} and for \eqref{eq:vf2}. As long as the discrete Lagrangian functions $L_d^1, L_d^2$ are consistent approximations to the continuous ones, the composed method has at least order one.

%

\subsection{Special classes: $S_1$ and $S_2$}
\label{type5}
While all the cases discussed above can be interpreted as splitting in two two-dimensional Hamiltonian systems, either solved by a symplectic method or turned into Lagrangian systems solved by a discrete Lagrangian method,
there is no obvious mechanical interpretation for cases $S_1$ and $S_2$. We are not aware of any numerical method for ordinary differential equations that is naturally related to these two generating forms in the same way as all the other cases discussed in this paper. In this respect, cases $S_1$ and case $S_2$ are novel cases.

Cases $S_1$ and $S_2$  are both associated to generating forms of type
\begin{displaymath}
	\lambda = \phi d x_1 + \Phi d X_1,
\end{displaymath}
which would suggest the choice of two vector fields of the form \eqref{eq:vf1}, clearly a degenerate and not particularly interesting vector field. 
Are there non-degenerate vector fields for which such generating form
gives consistent, non trivial maps? The answer is yes: 
it is possible to give explicit expressions for $\phi, \Phi$ so that the generating form $\lambda = \phi d x_1 + \Phi d X_1$ of cases $S_1$ and $S_2$ gives consistent, first order, volume preserving numerical methods at least in the case when the vector field $\mathbf{a}$ in \eqref{eq:vf} is linear: for \emph{linear vector fields}, $\phi, \Phi$ can be taken to be quadratic functions; thus the determining conditions in \eqref{eq:-o+general} are linear in the unknown variables.

Hereafter, we restrict our attention to linear divergence-free vector fields
\begin{align}
\label{linear}
\dot{x}_1&=a_1(x_1,x_2,x_3)=a_{11}x_1+a_{12}x_2+a_{13}x_3,\\
\dot{x}_2&=a_2(x_1,x_2,x_3)=a_{21}x_1+a_{22}x_2+a_{23}x_3, \qquad a_{11}+a_{22}+a_{33}=0. \\
\dot{x}_3&=a_3(x_1,x_2,x_3)=a_{31}x_1+a_{32}x_2+a_{33}x_3,
\end{align}

\subsubsection{Class $S_1$}
We have that $\Sigma(-)=1, \Sigma(\circ)=2,\Sigma(+)=3$, and $\tau$ is an even permutation. 
The generating one-form is 
\begin{displaymath}
	\lambda=\phi(x_1, x_2, X_1) dx_1+ \Phi (x_1, X_1, X_2) d X_1, 
\end{displaymath}
where $\phi, \Phi $ satisfy
\begin{align}
x_3&=\partial_{x_2}\phi(x_1, x_2, X_1),\\
\partial_{X_1} \phi(x_1, x_2, X_1) &= \partial_{x_1} \Phi (x_1,X_1,X_2),\\
X_3& = -\partial_{X_2} \Phi (x_1,X_1,X_2).
\label{eq:S1}
\end{align}
In addition to assuming linearity of the vector field, we also assume $a_{13}\neq0$ (a twist condition, implying that $x_3$ can be determined from $X_1$, given $x_1, x_2$).

Two possible choices of $(\phi, \Phi)$ yielding first order numerical methods of linear vector fields are given below.

\begin{proposition}
\label{th:S1}
Consider the  class $S_1$ generating one-form $\lambda = \phi dx_1 + \Phi dX_1$. Let
\begin{align}
	\phi(x_1,X_1,x_2)&=\frac{X_1-x_1-\Delta ta_{11}x_1-\Delta t ^2a_{21}a_{12}x_1\frac{k_2}{k_1}}{\Delta t a_{13}+\Delta t^2a_{23}a_{12}\frac{k_2}{k_1}}x_2-\frac{a_{12}}{k_1a_{13}+\Delta ta_{23}a_{12}k_2}\frac{x_2^2}{2},\\
	\Phi (x_1,X_1,X_2)&=-\frac{X_1-x_1-\Delta ta_{11}x_1}{\Delta ta_{13}}(1+\Delta ta_{33})X_2+X_2^2\big(-	\frac{\Delta ta_{32}}{2}+\frac{a_{12}}{2a_{13}}(1+\Delta ta_{33})\big)\\
&-\Delta t a_{31}X_1X_2-\frac{2X_1x_1-x_1^2(1+\Delta ta_{11})\Delta ta_{23}k_2+\Delta t^2a_{13}a_{21}k_2x_1^2}{2k_3},
\label{eq:S1_quispel}
\end{align}
where $k_1=1+\Delta t^2a_{11}a_{33}-\Delta ta_{22}$, $k_2=1+\Delta t^2\frac{a_{11}a_{33}}{1-\Delta ta_{22}}$, $k_3=\Delta ta_{13}(\Delta ta_{13}+\frac{k_2}{k_1}\Delta ta_{12}a_{23})$, and 
\begin{align}
\phi(x_1,X_1,x_2)&=\frac{X_1-(1+\Delta ta_{11})x_1}{\Delta a_{13}}x_2-\frac{1}{2}\frac{a_{12}}{a_{13}}x_2^2,\\
\Phi (x_1,X_1,X_2)&=-\frac{1}{l_1}\big((1+\Delta ta_{33})(\frac{X_1-(1+\Delta ta_{11})x_1}{\Delta ta_{33}}X_2-\frac{1}{2}\frac{a_{12}}{a_{13}}X_2^2)+\Delta ta_{31}x_1X_2\\
&+\frac{1}{2}\Delta ta_{32}X_2^2+\Delta t\frac{a_{12}}{a_{13}}(a_{21}x_1X_1+\frac{1}{2}a_{22}X_2^2)\big)\\
&+\Delta ta_{21}x_1^2/2+\Delta ta_{23}\frac{2X_1x_1-x_1^2(1+\Delta ta_{11})}{2\Delta ta_{13}},
\label{eq:S1_AZ}
\end{align} 
where  $l_1=1-\Delta t\frac{a_{12}}{a_{13}}a_{23}$.

Both choices \eqref{eq:S1_quispel} and \eqref{eq:S1_AZ} yield first-order volume preserving integrators for the vector field (\ref{linear}), provided that $a_{13}\not=0$.
\end{proposition}
The proof of the result can be found in Appendix~\ref{app:S1}.

\subsubsection{Class $S_2$}
This case corresponds to $\Sigma(-)=1, \Sigma(\circ)=3,\Sigma(+)=2$, with $\tau$ an odd permutation. 
The generating one-form is 
\begin{displaymath}
	\lambda=\phi(x_1, x_2, X_1) dx_1+ \Phi (x_1, X_1, X_3) d X_1, 
\end{displaymath}
where $\phi, \Phi $ satisfy
\begin{align}
x_3&=\partial_{x_2}\phi(x_1, x_2, X_1),\\
\partial_{X_1} \phi(x_1, x_2, X_1) &= \partial_{x_1} \Phi (x_1,X_1,X_3),\\
X_2& = \partial_{X_3} \Phi (x_1,X_1,X_3).
\label{eq:S2}
\end{align}
As for the $S_1$ case, we also assume $a_{12}\neq0$. Below we give the explicit expression of a choice $(\phi, \Phi)$ yielding first order numerical methods of linear vector fields. 

\begin{proposition}
\label{th:S2}
Consider the  class $S_2$ generating one-form $\lambda = \phi dx_1 + \Phi dX_1$. Let\begin{align}
\label{eq:S2_quispel}
	\phi(x_1,X_1,x_2)&=\frac{\big(m_1(X_1-x_1-\Delta ta_{11}x_1)-\Delta t^2a_{31}a_{13}m_2 x_1\big)x_2}{\Delta ta_{13}}   - \frac{ m_1a_{12}x_2^2}{2a_{13}}-\frac{\Delta ta_{32}m_2 x_2^2}2   \\
	\Phi(x_1,X_1,X_3)&= \frac{(1+\Delta ta_{22})\big(X_1-(1+\Delta ta_{11})x_1\big)X_3}{\Delta ta_{12}}-\frac{a_{13}(1+\Delta ta_{22})X_3^2}{2a_{12}}\\
& \quad \mbox{} +\Delta ta_{21}X_1X_3+\frac{\Delta ta_{23}}2 X_3^2+ \frac{m_1(2X_1x_1-\big(1+\Delta ta_{11})x_1^2\big)}{2\Delta t^2 a_{12}a_{13}},
\end{align}
where $m_1=1-\Delta t a_{33}+\Delta t^2a_{11}a_{22}$ and $m_2=1+\Delta t^2a_{11}a_{22}/(1-\Delta ta_{33})$.
The choice \eqref{eq:S1_quispel}  yields first-order volume preserving integrators for the vector field (\ref{linear}), provided that $a_{12}\not=0$.
\end{proposition}
The proof of this result is similar to that for \eqref{eq:S1_quispel}. For completeness, it can be found in Appendix~\ref{app:S2}.

\section{Conclusions and remarks}

In this paper, we have studied the thirty-six generating one-forms for volume preserving mappings in $\RR^3$. 
This is the first $n$-dimensional case for which there is a difference between area preservation, well understood using the tools of symplectic geometry and symplectic forms, and volume preservation.

By imposing equivalence relations (equivalence under relabeling and equivalence under adjunction), we have shown that all cases can be generated by five classes of differential one forms.

We have classified these five cases in terms of known numerical methods that preserve volume and denoted them as  SE+SE (already identified by Xue and Zanna in \cite{xue2014gf}), DL+SE, DL+DL, $S_1$ and $S_2$. Except for the special cases $S_1, S_2$, the classes can be naturally associated to the splitting of the vector field into two $2D$ Hamiltonian systems or into two Lagrangian system, solved by a symplectic method  (symplectic Euler, SE) or a discrete Lagrangian approach (DL), or both. 

Classes $S_1$ and $S_2$, both defined by a generating one-form of the type $\lambda = \phi dx_1 + \Phi d X_1$, are, to the best of our knowledge, novel cases. Whereas the other classes admit a natural mechanical interpretation (either Hamiltonian or Lagrangian mechanics), it is not clear whether there exists a natural mechanical interpretation for the classes $S_1$ and $S_2$. The corresponding generating forms can be used to generate well defined volume preserving maps, however, for general vector fields, these maps are hightly implicit and do not seem to lead to explicit or efficient numerical methods. For completeness, we have shown possible choices of functions $\phi, \Phi$, that yield consistent methods for linear vector fields. These two cases need a deeper understanding and will be the subject of future research.

\section*{Acknowledgement}
The work has been supported by NFR grant no.\ 191178/V30, under the project Geometric Numeric Integration in Applications,
by the SpadeACE~Project
and by the J.C.~Kempe memorial fund (grant no.~SMK-1238).

\appendix
\section{Appendix A}
\label{app:S1}
The two choices \eqref{eq:S1_quispel}-\eqref{eq:S1_AZ} correspond to two different techniques to determine $\phi, \Phi$. 

The first one, \eqref{eq:S1_quispel}, is based on the \emph{correction method} by Quispel \cite{GRW199526}: we determine three maps $f_1, f_2, f_3$ that give a volume preserving transformation. Thereafter we invert $f_1$ and use integration, differentiation and some other algebraic manipulations to obtain suitable $\phi, \Phi$. 

The second choice, \eqref{eq:S1_AZ}, is based on the following idea: choose a consistent method for $X_1$, depending on $x_3$, say for instance Forward Euler. Because of the linearity of the vector field, this always determines $x_3 = \phi_{x_2}(x_1,X_1,x_2)$, hence $\phi$, up to a function depending only on $x_1,X_1$. Next, we think of $x_3$ as a function of $x_2$.
Now, use a consistent map to obtain $X_3$, as function of $X_1, x_1, X_2$ and $x_3(x_1, X_1,s)$, where the occurrences of $x_2$ are replaced by $X_2$). Upon the replacement $x_2 \to X_2$ in $x_3$, the map is not necessarily consistent any longer, and some adjustments must be made, also to ensure consistency for $x_2$.

\begin{proof}\emph{[Prop.~\ref{th:S1}]}
We commence with the \eqref{eq:S1_quispel} case. Consider the implicit map 
\begin{align}
\label{ff}
X_1&=f_1(x_1,X_2,x_3),\\
x_2&=f_2(x_1,X_2,x_3),\\
X_3&=f_3(X_1,X_2,x_3),
\end{align}
and Quispel's correction method \cite{GRW199526}, to obtain
\begin{align}
\label{t5vp}
X_1&=x_1+\Delta t a_1(x_1,X_2,x_3),\\
X_2&=x_2+\Delta t a_2(x_1,X_2,x_3)- f_{correct}(x_1,X_2,x_3),\\
X_3&=x_3+\Delta t a_3(X_1,X_2,x_3),
\end{align}
where $f_{correct}$ is determined to obtain a volume preserving scheme,
\begin{align}
f_{correct}(x_1,X_2,x_3)&=\int^{X_2}_{const}\Delta t\frac{\partial a_3}{\partial x_3}(x_1+\Delta t a_1(x_1,X_2,x_3),X_2,x_3)-\Delta t\frac{\partial a_3}{\partial x_3}(x_1,X_2,x_3)\\
&+\Delta t^2\frac{\partial a_1}{\partial x_1}(x_1,X_2,x_3)\frac{\partial a_3}{\partial x_3}(x_1+\Delta t a_1(x_1,X_2,x_3),X_2,x_3)dX_2\\
&=\Delta t^2a_{11}a_{33}(X_2-const).
\end{align}
The integration constant should satisfy $const=\Delta t a_2(x_1,const,x_3)$\footnote{Due to consideration of consistency for $X_2$, see more details in \cite{GRW199526}.}, that is, 
\begin{equation*}
const=\frac{\Delta t a_{21}x_1+\Delta ta_{23}x_3}{1-\Delta ta_{22}}.
\end{equation*}
First of all, we calculate $X_2$ from the second equation of (\ref{t5vp}) which has the form,
\begin{equation}
\label{x2}
X_2=\frac{x_2+(\Delta ta_{21}x_1+\Delta ta_{23}x_3)k_2}{k_1},
\end{equation}
where $k_1=1+\Delta t^2a_{11}a_{33}-\Delta ta_{22}$ and $k_2=1+\Delta t^2\frac{a_{11}a_{33}}{1-\Delta ta_{22}}$.
Substituting the above equation into the first equation of (\ref{t5vp}), we can find $x_3$ in terms of variables $x_1, x_2, X_1$, denoted by $\tilde{x}_3(x_1,X_1,x_2)$. Substituting $\tilde{x}_3$ into the first equation of (\ref{eq:S1}) and integrating both sides, we obtain $\phi$
\begin{equation}
\phi=\frac{X_1-x_1-\Delta ta_{11}x_1-\Delta t ^2a_{21}a_{12}x_1\frac{k_2}{k_1}}{\Delta t a_{13}+\Delta t^2a_{23}a_{12}\frac{k_2}{k_1}}x_2-\frac{a_{12}}{k_1a_{13}+\Delta ta_{23}a_{12}k_2}\frac{x_2^2}{2}+\tilde{A}(x_1,X_1).
\end{equation}
 From the first equation of (\ref{t5vp}), we see that $x_3$ depends on the variables $x_1, X_1,X_2$. We then can solve $x_3$ in terms of  $x_1, X_1,X_2$ from that equation since we assume $a_{13}\neq0$, and denote  by $\hat{x}_3(x_1,x_2,X_1)$. We substitute $\hat{x}_3$ into the third equation of (\ref{t5vp}). From the third equations of (\ref{eq:S1}) and from  (\ref{t5vp}), we know that
\begin{equation*}
X_3(x_1,X_1,X_2)=\hat{x}_3(x_1,X_1,X_2)+\Delta t \mathbf{a}_3(X_1,X_2,\hat{x}_3(x_1,X_1,X_2))=-\partial_{X_2}\Phi .
\end{equation*}
From the above equation, we can integrate both sides with respect to $X_2$ and obtain
\begin{align}
\Phi &=-\frac{X_1-x_1-\Delta ta_{11}x_1}{\Delta ta_{13}}(1+\Delta ta_{33})X_2+X_2^2(-\frac{\Delta ta_{32}}{2}+\frac{a_{12}}{2a_{13}}(1+\Delta ta_{33}))\\
&-\Delta t a_{31}X_1X_2+\tilde{C}(x_1,X_1).
\end{align}
Without loss of generality, we set $\tilde{A}=0$. Using the second condition of (\ref{eq:S1}) , we obtain
\begin{equation*}
\tilde{C}_{x_1}=\frac{x_2}{\Delta a_{13}+\Delta ta_{12}a_{23}\frac{k_2}{k_1}}-\frac{(1+\Delta ta_{11})(1+\Delta ta_{33})}{ha_{13}}X_2,
\end{equation*}
From the first equation in (\ref{t5vp}), we have
\begin{equation*}
X_2=\frac{X_1-x_1-\Delta ta_{11}x_1-\Delta ta_{13}x_3}{\Delta ta_{12}},
\end{equation*}
By noticing the relation in (\ref{x2}), we obtain
\begin{equation*}
\tilde{C}_{x_1}=-\frac{X_1-x_1(1+\Delta ta_{11})\Delta ta_{23}k_2+\Delta t^2a_{13}a_{21}k_2x_1}{k_3},
\end{equation*}
where $k_3=\Delta ta_{13}(\Delta ta_{13}+\frac{k_2}{k_1}\Delta ta_{12}a_{23})$. 
From \eqref{t5vp} it is not difficult to see that (\ref{eq:S1_quispel}) gives a first order volume preserving method.

\vspace*{20pt}
Next, we consider the \eqref{eq:S1_AZ} case. Using the forward Euler method to solve the first equation of (\ref{linear}), that is
\begin{equation*}
X_1=(1+\Delta ta_{11})x_1+\Delta ta_{12}x_2+\Delta ta_{13}x_3.
\end{equation*}
Assuming that $a_{13}\neq0$ and solving for $x_3$, we have
\begin{equation}
\label{x3tilde}
x_3=\frac{X_1-(1+\Delta ta_{11})x_1-\Delta ta_{12}x_2}{\Delta ta_{13}}.
\end{equation}
Integrating both sides, we obtain
\begin{equation*}
\phi(x_1,X_1,x_2)=\frac{X_1-(1+\Delta ta_{11})x_1}{\Delta a_{13}}x_2-\frac{1}{2}\frac{a_{12}}{a_{13}}x_2^2+\tilde{A}(x_1,X_1),
\end{equation*}
where $\tilde{A}(x_1, X_1)$ is a function to be determined.

Using Euler method to solve the third equation of (\ref{linear}), where $x_3$ in (\ref{x3tilde}) is replaced by $x_3(x_1,X_1,X_2)$ and using $X_1=x_1+\Delta t a_1$ and $X_2=x_2+\Delta t a_2$, we obtain
\begin{align}
	X_3&=(1+\Delta ta_{33})\frac{X_1-(1+\Delta ta_{11})x_1-\Delta ta_{12}X_2}{\Delta ta_{13}}+\Delta 	ta_{31}x_1+a_{32}X_2\\
	&=(1+\Delta ta_{33})x_3+\Delta ta_{31}x_1+\Delta ta_{32}X_2-\Delta t\frac{a_{12}}{a_{13}}a_2+O(\Delta t^2).
\end{align}
There is a problem with the term $a_2$ (as in (\ref{linear})) in the above equation for consideration of consistency for $X_3$. So, we take $a_2=a_{21}x_1+a_{22}X_2+a_{23}X_3$ and substitute back into the above equation. Then, we obtain
\begin{equation*}
	(1-\Delta t\frac{a_{12}}{a_{13}}a_{23})X_3=(1+\Delta ta_{33})x_3+\Delta ta_{31}x_1+\Delta ta_{32}X_2-\Delta t	\frac{a_{12}}{a_{13}}(a_{21}x_1+a_{22}X_2)
\end{equation*}
Solving $X_3$ from the above equation, substituting it back into the third equation of (\ref{eq:S1}) and integrating both sides, we have
\begin{align}
	\Phi &=-\frac{1}{l_1}\big((1+\Delta ta_{33})(\frac{X_1-(1+\Delta ta_{11})x_1}{\Delta ta_{33}}X_2-\frac{1}{2}\frac{a_{12}}{a_{13}}X_2^2)+\Delta ta_{31}x_1X_2\\
	&+\frac{1}{2}\Delta ta_{32}X_2^2+\Delta t\frac{a_{12}}{a_{13}}(a_{21}x_1X_1+\frac{1}{2}a_{22}X_2^2)\big)+\tilde{C}(x_1,X_1),
\end{align}
where $l_1=1-\Delta t\frac{a_{12}}{a_{13}}a_{23}$ and $\tilde C$ is to be determined.
We have two functions $\tilde A, \tilde C$ to be determined and one equation (compatibility condition). The two functions are not independent, hence we set $\tilde{A}=0$. Using the compatibility condition $\partial_{X_1} \phi= \partial_{x_1} \Phi $, we obtain 
\begin{equation*}
\frac{x_2}{\Delta ta_{13}}=\frac{1}{l_1}(1+\Delta ta_{33})(1+\Delta ta_{11})X_2+\tilde{C}_{x_1}.
\end{equation*}
From the divergence-free condition $a_{22}=-(a_{11}+a_{33})$ we have
\begin{equation*}
X_2=x_2(1+\Delta ta_{22}+O(\Delta t^2))l_1+\tilde{C}_{x_1}.
\end{equation*}
Hence $X_2$ will be consistent provided that
\begin{equation}
	\tilde{C}=\Delta t a_{21}\frac{x_1^2}{2}+\Delta t a_{23}\frac{2X_1x_1-x_1^2(1+\Delta ta_{11})}{2\Delta ta_{13}}.
\end{equation}
From the proof, we see that (\ref{eq:S1_AZ}) gives a first order volume preserving method.
\end{proof}


\section{Appendix B}
\label{app:S2}


\begin{proof}\emph{[Prop.~\ref{th:S2}]}
The method generated by \eqref{eq:S2} and \eqref{eq:S2_quispel} is constructed in the same way as \eqref{eq:S1_quispel}. The implicit map reads
\begin{align}
\label{ffs2}
X_1&=f_1(x_1,x_2,X_3),\\
X_2&=f_2(X_1,X_2,x_3),\\
x_3&=f_3(x_1,x_2,X_3),
\end{align}
and the corrections method is 
\begin{align}
\label{t5vps2}
X_1&=x_1+\Delta t a_1(x_1,x_2,X_3),\\
X_2&=x_2+\Delta t a_2(X_1,x_2,X_3),\\
X_3&=x_3+\Delta t a_3(x_1,x_2,X_3)- f_{correct}(x_1,x_2,X_3),\\
\end{align}
where 
\begin{align}
	f_{correct}(x_1,x_2,X_3)&=\int^{X_3}_{const}\Delta t\frac{\partial a_2}{\partial x_2}(x_1+\Delta t 	a_1(x_1,x_2,X_3),x_2,X_3)-\Delta t\frac{\partial a_2}{\partial x_2}(x_1,x_2,X_3)\\
	&+\Delta t^2\frac{\partial a_1}{\partial x_1}(x_1,x_2,X_3)\frac{\partial a_2}{\partial x_2}(x_1+\Delta t 	a_1(x_1,x_2,X_3),x_2,X_3)dX_3\\
	&=\Delta t^2a_{11}a_{22}(X_3-const).
\end{align}
The integration constant should satisfy $const=\Delta t a_3(x_1,x_2,const1)$, that is, 
\begin{equation*}
	const1=\frac{\Delta t a_{31}x_1+\Delta ta_{32}x_2}{1-\Delta ta_{33}}.
\end{equation*}
The rest of the proof is similar to that of Prop.~\ref{th:S1}.

\end{proof}

\bibliographystyle{plain}
\bibliography{volumepreserving}

\begin{thebibliography}{10}

\bibitem{chartier2007preserving}
P.~Chartier and A.~Murua.
\newblock {Preserving first integrals and volume forms of additively split
  systems}.
\newblock {\em IMA Journal of Numerical Analysis}, 27(2):381--405, 2007.

\bibitem{F1986}
K.~Feng.
\newblock Difference schemes for hamiltonian formalism and symplectic geometry.
\newblock {\em J. Comput.\ Math}, 4(3):279--289, 1986.

\bibitem{kang95vpa}
Kang Feng and Zai~Jiu Shang.
\newblock Volume-preserving algorithms for source-free dynamical systems.
\newblock {\em Numer. Math.}, 71(4):451--463, 1995.

\bibitem{goldstein01cm}
H.~Goldstein, C.~P.~Poole Jr.\, and J.~L. Safko.
\newblock {\em Classical Mechanics}.
\newblock Pearson, 3rd edition, 2001.

\bibitem{hairer87sod}
E.~Hairer, S.~P. N{\o}rsett, and G.~Wanner.
\newblock {\em Solving {O}rdinary {D}ifferential {E}quations I. Nonstiff
  Problems}.
\newblock Springer-Verlag, Berlin, 2nd revised edition, 1993.

\bibitem{iserles07}
A.~Iserles, G.~R.~W. Quispel, and P.~S.~P. Tse.
\newblock B-series methods cannot be volume-preserving.
\newblock {\em BIT}, 47(2):351--378, 2007.

\bibitem{FW1989}
M.-Z.~Qin K.~Feng, H.M.~Wu and D.-L. Wang.
\newblock Construction of canonical difference schemes for {H}amiltonian
  formalism via generating functions.
\newblock {\em J. Comput.\ Math.}, 7:71--96, 1989.

\bibitem{lomeli09gff}
H.~E. Lomel{\'i} and J.~D. Meiss.
\newblock Generating forms for exact volume-preserving maps.
\newblock {\em Discrete and Continuous Dynamical Systems Series S},
  2(2):361--377, 2009.

\bibitem{marsden01}
J.~E. Marsden and M.~West.
\newblock Discrete mechanics and variational integrators.
\newblock {\em Acta Numerica}, 10:357--514, 2001.

\bibitem{mclachlan09evp}
R.~I. McLachlan, H.~Z. Munthe-Kaas, G.~R.~W. Quispel, and A.~Zanna.
\newblock Explicit volume-preserving splitting methods for linear and quadratic
  divergence-free vector fields.
\newblock {\em Found. Comput. Math.}, 8(3):335--355, 2008.

\bibitem{mclachlan04egi}
R.I. McLachlan and G.R.W. Quispel.
\newblock Explicit geometric integration of polynomial vector fields.
\newblock {\em BIT Numerical Mathematics}, 44:515--538, 2004.

\bibitem{McLachlan2002}
Robert~I. McLachlan and G.~Reinout~W. Quispel.
\newblock Splitting methods.
\newblock {\em Acta Numer.}, 11:341--434, 2002.

\bibitem{GRW199526}
G.~R.~W. Quispel.
\newblock Volume-preserving integrators.
\newblock {\em Phys. Lett. A}, 206(1-2):26--30, 1995.

\bibitem{quispel03evp}
G.R.W. Quispel and D.I. McLaren.
\newblock Explicit volume-preserving and symplectic integrators for
  trigonometric polynomial flows.
\newblock {\em J. Comp.\ Phys.}, 186(1):308--316, 2003.

\bibitem{shangXXgf1}
Zai-Jiu Shang.
\newblock Generating functions for {V}olume-preserving mappings with
  application {I}: Basic theory.
\newblock China/Korea Joint Seminar: Dynamical systems and their applications.

\bibitem{Shang1994}
Zai~Jiu Shang.
\newblock Construction of volume-preserving difference schemes for source-free
  systems via generating functions.
\newblock {\em J. Comput. Math.}, 12(3):265--272, 1994.

\bibitem{weyl40tmo}
H.~Weyl.
\newblock The method of orthogonal projection in potential theory.
\newblock {\em Duke Math.\ J.}, 7(1):411--444, 1940.

\bibitem{Xue:2012mf}
H.~Xue and A.~Zanna.
\newblock Explicit volume-preserving splitting methods for polynomial
  divergence-free vector fields.
\newblock {\em BIT Numerical Mathematics}, 53, 2012.

\bibitem{xue2014gf}
H.~Xue and A.~Zanna.
\newblock Generating functions and volume-preserving mappings.
\newblock {\em Discrete and Continuous Dynamical Systems Series A},
  34:1229--1249, 2014.

\bibitem{zanna2014tensor}
A.~Zanna.
\newblock Explicit volume-preserving splitting methods for divergence-free odes
  by tensor-produc basis decompositions.
\newblock {\em IMA J. Num.\ Anal.}, page To appear, 2013.

\end{thebibliography}

\end{document}